\documentclass[11pt]{article}

\usepackage[T1]{fontenc}
\usepackage{lmodern}
\usepackage[margin=1.02in]{geometry}
\usepackage{microtype}
\usepackage{amsmath,amssymb,amsthm,mathtools}
\usepackage{booktabs}
\usepackage{array}
\usepackage[hidelinks]{hyperref}
\hypersetup{
  pdftitle={Character Fourier Spectra of Circular Units and Twisted Bernoulli Class Components},
  pdfauthor={Peter Chocian},
  pdfkeywords={cyclotomic fields, circular units, generalized Bernoulli numbers, Kummer extensions, Hilbert class fields}
}

\allowdisplaybreaks
\numberwithin{equation}{section}

\newtheorem{theorem}{Theorem}[section]
\newtheorem{proposition}[theorem]{Proposition}
\newtheorem{corollary}[theorem]{Corollary}
\newtheorem{lemma}[theorem]{Lemma}
\theoremstyle{definition}

\theoremstyle{remark}
\newtheorem{remark}[theorem]{Remark}

\newcommand{\Q}{\mathbf Q}
\newcommand{\Z}{\mathbf Z}
\newcommand{\F}{\mathbf F}
\newcommand{\Gal}{\operatorname {Gal}}
\newcommand{\Cl}{\operatorname {Cl}}
\newcommand{\length}{\operatorname {length}}

\newcommand{\cO}{\mathcal O}

\title{Character Fourier Spectra of Circular Units\
and Twisted Bernoulli Class Components}
\author{Peter Chocian}
\date{29 July 2026}

\begin{document}
\maketitle

\begin{abstract}
Let \(\chi\) be a primitive odd Dirichlet character of conductor
\(f\).  Building on the universal projector polynomials \(P_m\)
introduced in the companion quadratic case
\cite{QuadraticCompanion}, we study their character Fourier
transform.  The polynomials are defined by
\[
\sum_{m\geq1}P_m(X)Y^m
 =-\log\bigl(1-X(1-e^{-Y})\bigr)
\]
For \(h_t=\zeta_f^t/(\zeta_f^t-1)\), we evaluate their character
Fourier spectrum.  For every odd \(m\geq1\),
\[
 \sum_{t\in(\Z/f\Z)^\times}\bar\chi(t)P_m(h_t)
 =\tau(\bar\chi)\frac{B_{m,\chi}}{m\,m!}.
\]
After reduction at any prime above \(p\nmid f\), the case
\(m=p-j\) identifies this spectrum, including its exact nonzero
scalar, with the divided generalized Bernoulli value attached to
\(\chi\omega^{-j}\).  Consequently the spectral-zero and
Bernoulli-zero criteria agree over every residue field; no splitting
hypothesis on the character coefficient field is required.

We connect the identity with the local Kummer spectrum of the
circular unit \(1-\zeta_f\zeta_p\), and make the global consequence
explicit in the first non-real case.  For the two primitive quartic
characters modulo \(5\) and primes \(p<500\), \(p\equiv1\pmod{20}\),
there are exactly eleven zero lines.  On each line an integral
character projection of \(1-\zeta_5\zeta_p\) is everywhere locally
unramified, and a finite split-prime Artin calculation proves that
it is not a global \(p\)-th power.  The associated generalized
Bernoulli value has exact valuation one, so the characterwise
abelian Main Conjecture shows that the radical generates the
complete order-\(p\) Hilbert class component.  Six of the eleven
components occur at classically regular primes.  A deterministic
integer-arithmetic program verifies the enumeration, the divided
digits, and every Artin certificate.
\end{abstract}



\section{Introduction}

Let \(\chi\) be a primitive odd Dirichlet character of conductor
\(f\).  Generalized Bernoulli values attached to \(\chi\) belong
simultaneously to three classical parts of cyclotomic arithmetic:
special values of Dirichlet \(L\)-functions, Stickelberger
annihilators, and character components of ideal class groups
\cite{Gras1977,Washington}.  Circular units provide the unit-side
counterpart, and the abelian Main Conjecture makes their relation
to class groups characterwise rather than merely numerical
\cite{MazurWiles,Sinnott,Greither}.

The quadratic family \(f=3\), using a reflected norm-one circular
unit, is treated in the companion paper
\cite{QuadraticCompanion}.  There the universal polynomials \(P_m\)
first appeared through the antisymmetric spectrum of that unit.
Here the mechanism is separated from quadratic reflection: we prove
the complete character Fourier transform for every primitive odd
character and then give its first non-real application at \(f=5\).

The purpose of this paper is to isolate an elementary transform
which links these two sides before any class-group computation is
made.  Define
\[
 h_t=\frac{\zeta_f^t}{\zeta_f^t-1}
 \qquad(t\in(\Z/f\Z)^\times).
\]
The local logarithm of \(1-\zeta_f^t\zeta_p\) is governed by one
sequence of polynomials \(P_m\), independent of \(f\), \(p\), and
\(\chi\).  The first main theorem evaluates their finite character
Fourier transform exactly:
\begin{equation}
\label{eq:intro-spectrum}
 \sum_t\bar\chi(t)P_m(h_t)
 =\tau(\bar\chi)\frac{B_{m,\chi}}{m\,m!}
 \qquad(m\text{ odd}).
\end{equation}
The proof is short: the projector generating function is summed
against \(\bar\chi\), a Gauss sum performs the Fourier transform,
and the result is the generalized Bernoulli generating function.
The point is not the separate ingredients, which are classical,
but the exact normalization forced by the local Kummer projector.

The residue-field form of \eqref{eq:intro-spectrum} is the useful
arithmetic statement.  If \(p\nmid f\), \(j\) is even, and
\(k=p-j\), then at every prime \(\mathfrak P\) above \(p\) in the
coefficient field,
\begin{equation}
\label{eq:intro-modp}
 \sum_t\bar\chi(t)P_k(h_t)
 =\frac{\tau(\bar\chi)}{f}(j-1)!\,b_{\chi,j}
 \quad\text{in }\kappa(\mathfrak P).
\end{equation}
Here \(b_{\chi,j}\) is the divided generalized Bernoulli value for
\(\chi\omega^{-j}\).  The factor multiplying it is a unit.  Thus
the local spectral zero and the Bernoulli zero are equivalent over
an arbitrary residue field; no assumption that the character
values already lie in \(\F_p\) is needed for the identity.

Equation \eqref{eq:intro-modp} turns a family of local expansions
into a character spectrometer.  Given \(f\), \(\chi\), and \(p\),
one may:
\begin{enumerate}
\item compute the local Fourier spectrum of
      \(1-\zeta_f\zeta_p\);
\item identify its blind lines with twisted Bernoulli zeros;
\item form an integral character projection of the circular unit;
\item test the projected global Kummer class at one completely
      split prime; and
\item compare the resulting unramified quotient with the full
      character component of the class group.
\end{enumerate}

The second main theorem carries this programme through at conductor
\(5\).  This is the first non-real coefficient family: the two
primitive characters are quartic and conjugate.  For
\(p<500\), \(p\equiv1\pmod{20}\), exactly eleven zero lines occur.
Each produces an explicit complete order-\(p\) Hilbert class
component.  Six rows occur at classically regular primes,
\[
 61\text{ (twice)},\qquad181\text{ (twice)},\qquad241,\qquad281.
\]
At \(p=61\), the two conjugate characters contribute on different
indices.  Hence conjugation does not make the two spectra redundant.

The universal theorem is formulated over coefficient fields.  The
global class-field application is deliberately stated only under
the conductor-five hypotheses used in its proof.  In particular,
we do not infer a uniform unramifiedness theorem at arbitrary
conductor from the Fourier identity alone.

\section{Universal projector polynomials}

For \(m\geq1\), define
\begin{equation}
\label{eq:Pm-definition}
 P_m(X)=
 \frac1{m!}
 \sum_{n=1}^{m}
 (-1)^{m+n}(n-1)!
 \left\{\begin{matrix}m\\n\end{matrix}\right\}X^n,
\end{equation}
where \(\left\{\begin{smallmatrix}m\\n\end{smallmatrix}\right\}\)
is a Stirling number of the second kind.  Thus
\[
 P_1(X)=X,\qquad
 P_2(X)=\frac{X^2-X}{2},\qquad
 P_3(X)=\frac{2X^3-3X^2+X}{6}.
\]

\begin{proposition}[Universal generating identity]
\label{prop:universal-generating}
The polynomials \(P_m\) satisfy
\begin{equation}
\label{eq:universal-generating}
 \boxed{\qquad
 \sum_{m\geq1}P_m(X)Y^m
 =-\log\bigl(1-X(1-e^{-Y})\bigr).
 \qquad}
\end{equation}
\end{proposition}

\begin{proof}
Expand the logarithm first:
\[
 -\log\bigl(1-X(1-e^{-Y})\bigr)
 =\sum_{n\geq1}\frac{X^n}{n}(1-e^{-Y})^n.
\]
The standard exponential generating function for Stirling numbers
gives
\[
 (1-e^{-Y})^n
 =n!\sum_{m\geq n}
 (-1)^{m+n}
 \left\{\begin{matrix}m\\n\end{matrix}\right\}
 \frac{Y^m}{m!}.
\]
Comparing the coefficient of \(Y^m\) gives
\eqref{eq:Pm-definition}.
\end{proof}

\begin{remark}
The same polynomials occur whenever a local factor is written as
\((1-h\pi)\) with \(\pi=1-\zeta_p\), and its logarithm is projected
onto a Teichm\"uller eigenspace.  This is the sense in which the
polynomials are universal: the conductor enters only through the
finite set of values \(h_t\).
\end{remark}

\section{The Fourier--Bernoulli identity over coefficient fields}

Let \(\chi\) be a primitive odd Dirichlet character of conductor
\(f>1\), extended by zero outside \((\Z/f\Z)^\times\).  Put
\[
 L=\Q\bigl(\chi,\zeta_f\bigr),\qquad
 \tau(\bar\chi)=
 \sum_{t\bmod f}\bar\chi(t)\zeta_f^t.
\]
For \(t\in(\Z/f\Z)^\times\), set
\[
 h_t=\frac{\zeta_f^t}{\zeta_f^t-1},
 \qquad
 A_{\chi,m}=
 \sum_{t\in(\Z/f\Z)^\times}
 \bar\chi(t)P_m(h_t).
\]

We normalize generalized Bernoulli numbers by
\begin{equation}
\label{eq:Bernoulli-generating}
 \sum_{a=1}^{f}\chi(a)
 \frac{Te^{aT}}{e^{fT}-1}
 =\sum_{m\geq0}B_{m,\chi}\frac{T^m}{m!}.
\end{equation}

\begin{theorem}[Characteristic-zero Fourier spectrum]
\label{thm:char-zero}
For every odd \(m\geq1\),
\begin{equation}
\label{eq:char-zero}
 \boxed{\qquad
 A_{\chi,m}
 =\tau(\bar\chi)\frac{B_{m,\chi}}{m\,m!}.
 \qquad}
\end{equation}
For even \(m\), both sides of the corresponding parity identity
vanish.
\end{theorem}

\begin{proof}
Sum \eqref{eq:universal-generating} against \(\bar\chi(t)\).
Since
\[
 1-h_t(1-e^{-Y})
 =\frac{1-\zeta_f^t e^{-Y}}{1-\zeta_f^t},
\]
differentiation gives
\begin{align*}
 \frac{d}{dY}\sum_{m\geq1}A_{\chi,m}Y^m
 &=-\sum_t\bar\chi(t)
   \frac{\zeta_f^t e^{-Y}}{1-\zeta_f^t e^{-Y}}\\
 &=-\sum_{n\geq1}e^{-nY}
   \sum_t\bar\chi(t)\zeta_f^{tn}\\
 &=-\tau(\bar\chi)
   \sum_{n\geq1}\chi(n)e^{-nY}.
\end{align*}
The Gauss-sum identity was used in the last line.  From
\eqref{eq:Bernoulli-generating}, evaluated at \(-Y\),
\[
 \sum_{n\geq1}\chi(n)e^{-nY}
 =\sum_{r\geq1}
 \frac{(-1)^rB_{r,\chi}}{r!}Y^{r-1}.
\]
Because \(\chi\) is odd, only odd \(r\) survive.  Comparing the
coefficient of \(Y^{m-1}\) gives
\[
 mA_{\chi,m}
 =\tau(\bar\chi)\frac{B_{m,\chi}}{m!},
\]
which is \eqref{eq:char-zero}.
\end{proof}

We now give the form that is stable under nonsplit coefficient
fields.  Let \(p\nmid f\) be an odd prime, let
\(\mathfrak P\) be a prime of \(L\) above \(p\), and write
\(\kappa_{\mathfrak P}=\cO_L/\mathfrak P\).  Reduction of
\(\zeta_f\) has exact order \(f\), so all \(h_t\) are defined in
\(\kappa_{\mathfrak P}\).  Let \(\omega\) denote the
Teichm\"uller character modulo \(p\).

For even \(j\), \(2\leq j\leq p-3\), define
\begin{equation}
\label{eq:divided-b}
 b_{\chi,j}
 =\overline{fB_{1,\chi\omega^{-j}}}
 \in\kappa_{\mathfrak P}.
\end{equation}
Equivalently, using Teichm\"uller lifts \(\widetilde a\),
\begin{equation}
\label{eq:divided-sum}
 b_{\chi,j}
 =\overline{
 \frac1p
 \sum_{\substack{1\leq a<fp\\(a,fp)=1}}
 a\,\chi(a)\widetilde a^{-j}}
 \quad\text{in }\kappa_{\mathfrak P}.
\end{equation}

\begin{theorem}[Residue-field Fourier--Bernoulli identity]
\label{thm:residue-field}
Let \(j\) be even with \(2\leq j\leq p-3\), and put
\(k=p-j\).  Then
\begin{equation}
\label{eq:residue-field}
 \boxed{\qquad
 \overline{A_{\chi,k}}
 =\frac{\overline{\tau(\bar\chi)}}{\bar f}
 (j-1)!\,b_{\chi,j}
 \quad\text{in }\kappa_{\mathfrak P}.
 \qquad}
\end{equation}
The scalar \(\overline{\tau(\bar\chi)}/\bar f\) is nonzero.
Consequently
\[
 \overline{A_{\chi,k}}=0
 \quad\Longleftrightarrow\quad
 b_{\chi,j}=0
\]
at every choice of \(\mathfrak P\).
\end{theorem}

\begin{proof}
The generalized Kummer congruence gives
\[
 \frac{B_{k,\chi}}{k}
 \equiv B_{1,\chi\omega^{-j}}
 =\frac{b_{\chi,j}}f
 \pmod{\mathfrak P},
\]
because \(k-1\equiv-j\pmod{p-1}\).  Wilson's congruence and the
parity of \(j\) give
\[
 (p-j)!(j-1)!\equiv1\pmod p.
\]
Substitution in Theorem~\ref{thm:char-zero} proves
\eqref{eq:residue-field}.  Finally,
\[
 \tau(\chi)\tau(\bar\chi)=\chi(-1)f=-f,
\]
so both Gauss sums are units at every prime above \(p\nmid f\).
\end{proof}

\begin{remark}
When \(p\) splits in the coefficient field and the selected
\(f\)-th roots of unity lie in \(\F_p\),
Theorem~\ref{thm:residue-field} is an identity in \(\F_p\).  In
general it lives in a finite extension.  At each prime
\(\mathfrak P\), the spectral and Bernoulli zero criteria still
coincide.  Different primes above \(p\) may represent conjugate
coefficient embeddings and need not select the same index.
\end{remark}

\section{Cyclotomic interpretation of the spectrum}

Let
\[
 K_{f,p}=\Q(\zeta_f,\zeta_p),\qquad
 \eta=\zeta_p,\qquad
 u_{f,p}=1-\zeta_f\eta,
\]
and choose a prime above \(p\).  If \(\xi\) denotes the reduction
of \(\zeta_f\) and \(\pi=1-\eta\), then
\begin{equation}
\label{eq:local-factor}
 1-\xi^t\eta
 =(1-\xi^t)(1-h_t\pi),
 \qquad
 h_t=\frac{\xi^t}{\xi^t-1}.
\end{equation}

\begin{proposition}[Local spectral meaning]
\label{prop:local-spectrum}
Let \(k=p-j\) be as in Theorem~\ref{thm:residue-field}.  Up to a
nonzero normalization scalar, the first possible local coordinate
of the \(\chi\omega^k\)-projection of \(u_{f,p}\) is
\[
 A_{\chi,k}
 =\sum_{t\in(\Z/f\Z)^\times}
 \bar\chi(t)P_k(h_t).
\]
Thus the local coordinate vanishes if and only if
\(b_{\chi,j}=0\).
\end{proposition}

\begin{proof}
The nontrivial \(\omega^k\)-projection kills the constant factor
in \eqref{eq:local-factor}.  Apply the local logarithm to
\(1-h_t\pi\), project onto the \(\omega^k\)-eigenspace, and then
sum the conductor components against \(\bar\chi(t)\).  The
coefficient produced by the local projector is exactly
\(P_k(h_t)\) by Proposition~\ref{prop:universal-generating}.
The last assertion is Theorem~\ref{thm:residue-field}.
\end{proof}

\begin{remark}[The global boundary]
Proposition~\ref{prop:local-spectrum} is uniform in \(f\), but a
global unramified Kummer extension requires more: an integral
projected representative, control at every prime above \(p\), and
control at primes dividing \(f\).  These conditions are proved
below for \(f=5\).  No general-conductor class-field assertion is
being smuggled into the Fourier identity.
\end{remark}

\section{The conductor-five projected unit}

Henceforth assume
\[
 p\equiv1\pmod{20},\qquad
 K_p=\Q(\zeta_{5p}),\qquad
 \zeta=\zeta_5,
 \qquad u=1-\zeta\eta.
\]
The two primitive quartic characters modulo \(5\) are conjugate.
At a chosen \(p\)-adic embedding let \(\iota_p\) be the smaller
root of \(X^2+1\) in \(\F_p\), and put
\[
 \chi_+(2)=\iota_p,
 \qquad
 \chi_-(2)=-\iota_p.
\]

\begin{lemma}[A genuine circular unit]
\label{lem:global-unit}
The element \(u\) is a global unit and
\[
 N_{K_p/\Q(\zeta)}(u)=1.
\]
\end{lemma}

\begin{proof}
The product \(\zeta\eta\) is a primitive \(5p\)-th root of
unity.  Since \(5p\) has two distinct prime factors,
\(\Phi_{5p}(1)=1\), and hence \(u\) is a unit.  Moreover,
\[
 \prod_{r=1}^{p-1}(1-\zeta\eta^r)
 =\frac{1-\zeta^p}{1-\zeta}=1,
\]
because \(p\equiv1\pmod5\).  The omitted \(r=0\) factor is
essential in the quotient.
\end{proof}

Fix \(\chi\in\{\chi_+,\chi_-\}\).  For even \(j\), put
\[
 k=p-j,\qquad
 \theta=\chi\omega^k,
 \qquad
 \psi=\bar\chi\omega^j.
\]
Then \(\psi=\omega\theta^{-1}\).  For
\(a\in(\Z/5p\Z)^\times\), choose
\[
 c_a\in\{0,\ldots,p-1\},\qquad
 c_a\equiv
 \bigl(4(p-1)\bigr)^{-1}\theta(a)^{-1}\pmod p,
\]
and define
\begin{equation}
\label{eq:projected-unit}
 \alpha_{\chi,k}
 =\prod_{a\in(\Z/5p\Z)^\times}\sigma_a(u)^{c_a}.
\end{equation}

\begin{proposition}[Integral projection]
\label{prop:idempotent}
The class of \(\alpha_{\chi,k}\) in
\(K_p^\times/K_p^{\times p}\) is the \(\theta\)-projection of
\(u\).  It is independent of the selected integer lifts \(c_a\).
\end{proposition}

\begin{proof}
The coefficients lift the idempotent
\[
 e_\theta=
 \frac1{4(p-1)}
 \sum_{a\in(\Z/5p\Z)^\times}
 \theta(a)^{-1}\sigma_a
 \quad\text{in }\F_p[\Gal(K_p/\Q)].
\]
Changing any coefficient by a multiple of \(p\) changes
\eqref{eq:projected-unit} by a global \(p\)-th power.
\end{proof}

There are four primes of \(K_p\) above \(p\), corresponding to
the four reductions \(\xi^b\) of \(\zeta_5\).  At the prime
corresponding to \(\xi^b\), substitution \(t\mapsto bt\) multiplies
the local Fourier coordinate by \(\chi(b)\).  Hence vanishing at
one of the four primes is equivalent to vanishing at all four.

\begin{corollary}[Unramified radical]
\label{cor:unramified}
If \(b_{\chi,j}=0\), then \(\alpha_{\chi,k}\) is a local
\(p\)-th power at every prime above \(p\), and
\[
 K_p(\alpha_{\chi,k}^{1/p})/K_p
\]
is everywhere unramified.
\end{corollary}

\begin{proof}
Theorem~\ref{thm:residue-field} and
Proposition~\ref{prop:local-spectrum} make the local
\(\omega^k\)-coordinate vanish at every prime above \(p\).
Since \(3\leq k\leq p-2\), the one-coordinate local Kummer
criterion makes the projected class a local \(p\)-th power there.
At every finite prime away from \(p\), the defining element is a
global unit, so its Kummer extension is unramified.
\end{proof}

For computation one can compress the divided Bernoulli sum.

\begin{lemma}[Short finite form]
\label{lem:short-form}
Let \(s=\chi(2)\in\F_p\).  For even \(j\) in the stated range,
\[
 \boxed{\qquad
 b_{\chi,j}
 =\sum_{r=1}^{p-1}W_{r\bmod5}(s)r^{-j},
 \qquad}
\]
where
\[
 (W_0,W_1,W_2,W_3,W_4)
 =(-3-s,-3-s,2-s,2+4s,2-s).
\]
\end{lemma}

\begin{proof}
For fixed \(r\in\{1,\ldots,p-1\}\), group the five terms
\(r+mp\), \(0\leq m<5\).  Exactly one is divisible by \(5\),
and the other four have the same Teichm\"uller factor.  Since
\(p\equiv1\pmod5\), their character values are \(\chi(r+m)\).
The four nonzero values sum to zero, so
\[
 \sum_m(r+mp)\chi(r+m)
 =p\sum_m m\chi(r+m).
\]
Evaluating the final four-term sum for each residue class of
\(r\pmod5\) gives the displayed vector.
\end{proof}

\section{The complete conductor-five calculation below 500}

\begin{theorem}[Conductor-five class components]
\label{thm:catalogue}
For \(p<500\), \(p\equiv1\pmod{20}\), and
\(\chi\in\{\chi_+,\chi_-\}\), the equality
\(b_{\chi,j}=0\) holds precisely on the following eleven lines:
\[
\begin{array}{c@{\qquad}c@{\qquad}r@{\qquad}r}
p&\chi&j&k=p-j\\ \toprule
61 &+&34&27\\
61 &-&10&51\\
101&+&92&9\\
181&+&104&77\\
181&+&138&43\\
241&-&196&45\\
281&+&46&235\\
421&+&50&371\\
421&-&90&331\\
461&+&2&459\\
461&+&376&85
\end{array}
\]
For every row, the extension
\[
 K_p(\alpha_{\chi,k}^{1/p})/K_p
\]
is nontrivial, everywhere unramified, and is the complete
\(\psi=\bar\chi\omega^j\) component of the Hilbert class field.
Moreover,
\[
 \boxed{\qquad
 \#\bigl(\Cl(K_p)\otimes\Z_p\bigr)_\psi=p.
 \qquad}
\]
Every local and generalized Bernoulli zero in the table is simple.
\end{theorem}

The proof occupies the next three sections.  Its three finite
layers are logically separate: exhaustive enumeration, a
second-order valuation calculation, and a nontrivial Artin symbol.

The ordinary irregular indices at the seven supporting primes are
\[
\begin{array}{c@{\qquad}l}
p&\text{ordinary irregular indices}\\ \toprule
61&\varnothing\\
101&68\\
181&\varnothing\\
241&\varnothing\\
281&\varnothing\\
421&240\\
461&196.
\end{array}
\]
Consequently six of the eleven rows lie at classically regular
primes.  Ordinary irregularity neither implies nor is implied by
the non-real twisted degeneracy in this range.

\section{Second-order digits}

For each row, lift the selected fifth root and \(\chi(2)\)
Teichm\"uller-wise to \(\Z/p^2\Z\), and set
\[
 S_{\chi,k}=
 \sum_{t=1}^{4}\bar\chi(t)P_k(h_t)\pmod{p^2}.
\]
Using Teichm\"uller lifts modulo \(p^3\), put
\[
 T_{\chi,j}=
 \sum_{\substack{1\leq a<5p\\(a,5p)=1}}
 a\,\chi(a)\widehat a^{-j}\pmod{p^3}.
\]

\begin{proposition}[Exact divided digits]
\label{prop:digits}
The exact divided digits are
\[
\begin{array}{c@{\;}c@{\;}r@{\;}r@{\;}r@{\;}r}
p&\chi&j&S_{\chi,k}/p&T_{\chi,j}/p^2&B/p\\ \toprule
61 &+&34 &33 &11 &51\\
61 &-&10 &12 &3  &25\\
101&+&92 &92 &70 &14\\
181&+&104&127&67 &122\\
181&+&138&95 &14 &39\\
241&-&196&235&152&175\\
281&+&46 &153&145&29\\
421&+&50 &126&230&46\\
421&-&90 &91 &379&160\\
461&+&2  &437&156&400\\
461&+&376&265&261&421
\end{array}
\]
where
\[
 B=B_{1,\chi\omega^{-j}}=\frac{T_{\chi,j}}{5p}.
\]
Thus, for every row,
\[
 v_p(S_{\chi,k})=1,
 \qquad
 v_p(T_{\chi,j})=2,
 \qquad
 v_p(B_{1,\chi\omega^{-j}})=1.
\]
\end{proposition}

\begin{proof}
The entries are direct exact evaluations in \(\Z/p^2\Z\) and
\(\Z/p^3\Z\).  All displayed divided digits are nonzero.  Division
by \(5p\) lowers the valuation by one and multiplies the final
digit by \(5^{-1}\pmod p\).
\end{proof}

\section{Finite Artin certificates}

Let \(q\equiv1\pmod{5p}\) be prime, and let
\(\Xi\in\F_q^\times\) have exact order \(5p\).  Put
\[
 \eta_q=\Xi^5,\qquad \zeta_q=\Xi^p.
\]
Reduction of \eqref{eq:projected-unit} at
\((q,\zeta_{5p}-\Xi)\) is
\[
 M_{p,\chi,j}
 =\prod_{a\in(\Z/5p\Z)^\times}
 (1-\zeta_q^a\eta_q^a)^{c_a}.
\]

\begin{proposition}[Certificate table]
\label{prop:certificates}
For the following data,
\[
 M_{p,\chi,j}^{(q-1)/p}=\eta_q^e\neq1:
\]
\[
\begin{array}{c@{\;}c@{\;}r@{\quad}r@{\quad}r@{\quad}r@{\quad}r}
p&\chi&j&q&\Xi&M&e\\ \toprule
61 &+&34 &1831&64  &334 &25\\
61 &-&10 &1831&64  &1258&13\\
101&+&92 &5051&1024&1710&31\\
181&+&104&1811&9   &438 &74\\
181&+&138&1811&9   &829 &89\\
241&-&196&2411&9   &1465&5\\
281&+&46 &8431&64  &8096&168\\
421&+&50 &4211&9   &2467&151\\
421&-&90 &4211&9   &1706&253\\
461&+&2  &9221&16  &5130&400\\
461&+&376&9221&16  &3883&241
\end{array}
\]
Hence every \(\alpha_{\chi,k}\) in
Theorem~\ref{thm:catalogue} is a nontrivial global Kummer class.
\end{proposition}

\begin{proof}
In each row, \(q\) is prime, \(\Xi\) has exact order \(5p\), and
direct evaluation gives \(M\).  Its \(p\)-th power residue symbol
is the displayed nonzero power of \(\eta_q\).  A global
\(p\)-th power would have trivial symbol at every such prime.
\end{proof}

\begin{corollary}
\label{cor:degree}
Each row defines a nontrivial everywhere-unramified cyclic
extension of degree \(p\), with class character
\[
 \omega\theta^{-1}=\bar\chi\omega^j=\psi.
\]
\end{corollary}

\section{Completeness of the class components}

\begin{theorem}[Class-side completeness]
\label{thm:complete}
For every row of Theorem~\ref{thm:catalogue},
\[
 \length_{\Z_p}
 \bigl(\Cl(K_p)\otimes\Z_p\bigr)_\psi
 =v_p(B_{1,\psi^{-1}})=1.
\]
Therefore \(\alpha_{\chi,k}^{1/p}\) generates the complete
\(\psi\)-component of the Hilbert class field.
\end{theorem}

\begin{proof}
The character \(\psi=\bar\chi\omega^j\) is odd, primitive of
conductor \(5p\), and distinct from \(\omega\).  Moreover,
\[
p\nmid[K_p:\Q]=4(p-1),
\]
so the character decomposition is semisimple.  In this semi-simple
setting the algebraic and arithmetic isotypic components of the
class group coincide; the distinction, essential when \(p\) divides
the degree, is analyzed in \cite{GrasPhi}.

The finite-level
odd-character consequence of the abelian Main Conjecture
\cite{MazurWiles,Greither} gives the characterwise equality
\[
 \length_{\Z_p}
 \bigl(\Cl(K_p)\otimes\Z_p\bigr)_\psi
 =v_p(B_{1,\psi^{-1}}).
\]
This is a per-character statement; the analytic relative class
number formula alone controls only the product over the odd
characters.  Proposition~\ref{prop:digits} makes the valuation
one, while Corollary~\ref{cor:degree} supplies a nontrivial
unramified quotient of degree \(p\).  The quotient therefore
exhausts the component.
\end{proof}

The exhaustive enumeration follows from the short sum in
Lemma~\ref{lem:short-form}; Proposition~\ref{prop:digits} proves
simplicity; Proposition~\ref{prop:certificates} proves global
nontriviality; and Theorem~\ref{thm:complete} proves completeness.
This completes the proof of Theorem~\ref{thm:catalogue}.

\section{The first conjugate pair at p = 61}

At \(p=61\), both coefficient embeddings contribute, but on
different spectral lines:
\[
 (\chi_+,j,k)=(\chi_+,34,27),
 \qquad
 (\chi_-,j,k)=(\chi_-,10,51).
\]
For the canonical fifth root \(\xi=9\pmod{61}\), its lift modulo
\(61^2\) is \(3120\).  The two lifts of \(\chi(2)\) are
\[
 682,\qquad3039.
\]
The local digits are
\[
 S_{\chi_+,27}=61\cdot33,
 \qquad
 S_{\chi_-,51}=61\cdot12,
\]
and the class digits are
\[
 \frac{B_{1,\chi_+\omega^{-34}}}{61}\equiv51,
 \qquad
 \frac{B_{1,\chi_-\omega^{-10}}}{61}\equiv25
 \pmod{61}.
\]

At \(q=1831\), \(\Xi=64\) has exact order
\(305=5\cdot61\).  The two projected products are \(334\) and
\(1258\), with
\[
 334^{30}=\eta_q^{25}\neq1,
 \qquad
 1258^{30}=\eta_q^{13}\neq1.
\]
Thus \(K_{61}\) has two explicitly generated order-\(61\) class
components on distinct non-real character lines, although \(61\)
is classically regular.  This example answers a basic structural
question: conjugate characters need not vanish at conjugate copies
of the same index.

\section{Reproducibility}

The ancillary program
\[
 \texttt{anc/verify\_conductor5\_twisted\_class\_generators.py}
\]
uses only the Python standard library and exact integer arithmetic.
It checks:
\begin{itemize}
\item every prime \(p<500\) with \(p\equiv1\pmod{20}\);
\item both embeddings of the quartic character;
\item the exact eleven-line enumeration;
\item an independent evaluation through
      Lemma~\ref{lem:short-form};
\item the scalar \(\tau(\bar\chi)/5\) in
      Theorem~\ref{thm:residue-field};
\item all local and generalized Bernoulli second-order digits;
\item the ordinary irregular indices and the regular-row count;
\item primality and exact root order for every certificate; and
\item all eleven projected products and residue-symbol exponents.
\end{itemize}

The universal identity itself can be checked independently at a
chosen residue prime by evaluating both sides of
\eqref{eq:residue-field}; the proof of
Theorem~\ref{thm:char-zero}, however, is formal and does not depend
on the finite catalogue.

\section{Scope, relation to prior work, and further questions}

Generalized Bernoulli numbers, Gauss sums, Kummer congruences,
circular units, and character components of cyclotomic class
groups are classical.  The Bernoulli/class-group relation belongs
to the Stickelberger and Herbrand tradition
\cite{Gras1977,Washington}; the global unit and class-index setting
belongs to the circular-unit tradition
\cite{Sinnott,Greither}; and the characterwise completeness used
above is a finite-level consequence of the abelian Main Conjecture
\cite{MazurWiles,Greither}.

The contribution proposed here is the combined mechanism:
the universal local projector \(P_m\), its exactly normalized
character Fourier transform, the residue-field zero criterion,
and the conversion of that zero into a certificate-checkable
unramified radical.  The conductor-five theorem additionally shows
that this mechanism is not an artifact of quadratic reflection or
of a norm-one quotient, and that conjugate non-real characters can
select different indices.

The natural next problem has two parts.  Analytically, determine
the distribution and depth of the zeros \(b_{\chi,j}\) as \(f\),
\(\chi\), and \(p\) vary.  Arithmetically, determine hypotheses
under which the integral projection of \(1-\zeta_f\zeta_p\) gives
the full predicted class component for arbitrary conductor.  The
Fourier identity solves neither local descent at conductor primes
nor the reflected unit/class fork by itself; it gives a uniform
coordinate in which those questions can be asked and computed.
A systematic multi-conductor survey and its distributional
statistics will appear separately.

\medskip
\noindent
\textbf{Originality status.}
The exact projector normalization and its use as a uniform explicit
Kummer instrument appear not to be recorded in the sources checked
so far.  This paper does not claim a definitive priority result.
A specialist comparison with the literature on cyclotomic-unit
relations, explicit Jacobi-sum radicals, and computational class
groups of composite cyclotomic conductor remains appropriate before
submission.

\section*{Acknowledgements}

Computational exploration, drafting, and verification were assisted
by OpenAI's 5.6 Sol and Anthropic's Fable 5.  Responsibility for the mathematical statements
and final presentation remains with the author.

\end{document}